\documentclass[11pt,letterpaper]{amsart}

\usepackage{amsmath,amsthm,amssymb,amsfonts}
\usepackage{mathrsfs}     
\usepackage{hyperref}     
\usepackage{graphicx}     
\usepackage{enumerate}    
\RequirePackage{rotating}

\usepackage{xcolor}
\usepackage{algorithm}
\usepackage{graphicx}
\usepackage{mathtools}
\usepackage{mathrsfs}
\usepackage{fix-cm}
\usepackage{multirow}
\usepackage{amsmath}
\usepackage{amssymb}
\usepackage{latexsym}
\usepackage{dsfont}
\usepackage{xcolor}
\usepackage{cite}
\usepackage{dsfont}
\usepackage{enumitem}
\usepackage{todonotes}
\usepackage{hyperref}
\usepackage{booktabs}
\usepackage{bbm}
\usepackage{framed}
\usepackage{mdframed}
\usepackage{pgfplots}
\pgfplotsset{compat = newest}
\usepackage{caption}
\usepackage{subcaption}
\usepackage{color}
\usepackage{graphicx,epsfig}
\usepackage{verbatim}
 
\usepackage{hyperref}

\usepackage{tabularray}
\UseTblrLibrary{siunitx}

\hypersetup{
colorlinks=true,
linkcolor=blue,
citecolor=darkgreen,
urlcolor=blue}
\definecolor{darkgreen}{RGB}{0,150,0}
\definecolor{C0}{RGB}{31,119,180}
\definecolor{C1}{RGB}{255, 127, 14}
\definecolor{MineShaft}{rgb}{0.188,0.188,0.188}
\definecolor{red}{RGB}{0,0,0}
\usepackage{algpseudocode}

\numberwithin{equation}{section}
\newcommand{\inner}[2]{\langle #1, #2 \rangle}

\newcommand{\R}{\mathbb{R}}
\newcommand{\N}{\mathbb{N}}

\newcommand{\inte}[1]{{\rm int (#1)}\kern 0.12em}
\newcommand{\cl}[1]{\overline{#1} }

\newcommand{\orbit}[1]{{\rm orb(#1)}\kern 0.12em}

\let\epsilon\varepsilon
\let\subseteq\subset 

\theoremstyle{plain}
\newtheorem{theorem}{Theorem}[section]
\newtheorem{lemma}[theorem]{Lemma}
\newtheorem{proposition}[theorem]{Proposition}
\newtheorem{corollary}[theorem]{Corollary}

\theoremstyle{definition}
\newtheorem{definition}[theorem]{Definition}
\newtheorem{example}[theorem]{Example}

\theoremstyle{remark}
\newtheorem{remark}[theorem]{Remark}

\numberwithin{equation}{section}
\newcounter{claim}[theorem]

\title[On the topological dynamical system of tilings]{On the topological dynamical system of tilings}

\author{José Pablo Santander}
\address{Center for Mathematical Modeling (CNRS IRL2807) , Universidad de Chile, Santiago, Chile}
\email{josesantander@ug.chile.cl}

\thanks{ 
            This work was partially supported by Centro de
			Modelamiento Matem\'{a}tico (CMM), ACE210010 and FB210005, BASAL funds for
			center of excellence and ANID-Chile grant: Fondecyt Regular 1240335
}
\begin{document}

\begin{abstract}

Tilings can be studied from the perspective of dynamical systems and ergodic theory. Some fundamental features of tiling spaces as topological dynamical systems, such as the continuity of the action by translation, the compactness with the tiling metric, and the minimality of the system generated by a repetitive tiling, are among the most frequently mentioned basic facts in the dynamical theory of aperiodic order. Nonetheless, despite their classical status, complete proofs of these facts, together with the preliminary material they rely on, to the best of our knowledge, are not available in the literature. 
The purpose of this survey is to give a self-contained, rigorous, and pedagogical treatment of these foundations with detailed preliminary material. We hope that this exposition will provide a useful reference and contribute to a clearer and more accessible understanding of the fundamentals of the theory of tiling dynamical systems.

\end{abstract}

\maketitle


\section{Introduction}

In a few words, a tiling is an arrangement of “pieces” or “tiles” (which are copies of some general pieces called prototiles) that cover Euclidean space without overlapping.

One could say that the modern mathematical theory of tilings begins in the 1960s with the work of Hao Wang \cite{Wang1961}, from the perspective of computability theory and logic.

In \cite{Berger1966}, Robert Berger proved that the problem of determining whether a given set of tiles can tile the plane is undecidable (see also \cite{Robinson1971}). A central ingredient in the proof is the construction of a set of tiles that can tile the plane only in an \emph{aperiodic} way. Such sets of tiles are called \emph{aperiodic}. 
A tiling is said to be \emph{aperiodic} if no nontrivial translation of it coincides with the original tiling.

On the other hand, \textbf{Crystallography} is the study of the atomic structure of crystals through the commonly used technique of X-ray diffraction. For a long time, only materials whose diffraction patterns produced “ordered” structures exhibiting translational symmetries throughout space were known. It was not until 1982 that the materials scientist Dan Shechtman observed certain aluminum--manganese alloys whose diffraction patterns exhibited rotational symmetries of order 5, which cannot occur in patterns with translational symmetries (by the crystallographic restriction theorem). The discovery of these materials, whose structural form is ordered but non-periodic (as they lack translational symmetry) earned him the Nobel Prize in Chemistry in 2011.

To some extent, the physical discovery of these materials,  known as quasicrystals, had been anticipated by the appearance in the literature of certain aperiodic tilings. The British physicist Sir Roger Penrose published examples of tilings exhibiting diffraction patterns while lacking this type of translational symmetry \cite{Penrose1974Aesthetics}.

After the discovery of quasicrystals, aperiodic tilings began to be studied from the perspective of ergodic theory, especially in connection with their use as models for physical systems. This marked the beginning of a more dynamical approach.
The analogy between tilings and dynamical systems on $\mathbb{Z}^d$ was observed in \cite{Rudolph1989}, and the application of ergodic theory to tiling theory was reviewed in \cite{Radin1996}.

One of the fundamental steps to obtain a dynamical system with desirable properties is to endow the space with a topology with respect to which is compact. The first notion of a topology on a space of tilings comes from Rudolph in \cite{Rudolph1989} and it was refined in \cite{AndersonPutnam1998}. Other topology was introduced in \cite{RadinWolff1992} (see also \cite{Radin1993, BerendRadin1993, HoltonRadinSadun2005 , RadinSadun2001}) and it was proved that such topology makes the space of tilings a compact metric space without the requirement of finite local complexity , however, the topology \eqref{def_metrica} defined in \cite{AndersonPutnam1998}  is the most usual for tiling dynamics, since it makes tiling dynamical systems a generalization of shift spaces in symbolic dynamics. The compactness of tiling spaces (under the hypotheses of finite local complexity) with its usual metric is a well-known fact, and \cite{Rudolph1989} is the oldest reference for the compactness of the system generated by a tiling under suitable hypotheses, nonetheless, the result in \cite{Rudolph1989} is mentioned without a proof in a context that appears, at first glance, to be not identical to finite local complexity, but it is pointed out in \cite{Robinson1996dynamical} that it is a particular case of finite local complexity. Nevertheless, as it is remarked in \cite{Robinson2004}, the details are somewhat subtle, in fact, it is difficult to find in the literature a complete account of these arguments; most standard references briefly state that compactness follows from a diagonal argument (see \cite{Sadun2006,Solomyak1997}) without providing further details, and in some cases, completeness is stated \cite{Robinson2004, bufetov2013limit}, or assumed beforehand \cite{Sadun2008_topology_of_tilings}. Some authors consider the hull of a tiling (the dynamical system generated by a tiling) as the metric completion of the orbit of such tiling under the metric \eqref{def_metrica} \cite{Trevino2023, trevino2025quantitative, liu2026new, KellendonkPutnam2000, Putnam2000ordered}.

An important property of dynamical systems is minimality. The first work to relate minimality with repetitiveness was \cite{RadinWolff1992}, where it was proved that, under the topology defined there, minimality implies repetitiveness. Then, it was proved in \cite{AndersonPutnam1998} for substitution tilings that, under the metric \eqref{def_metrica}, the action of $\R^d$ by translation is minimal. The reference that is usually cited for the equivalence between the minimality of the dynamical system generated by a tiling  and the repetitiveness of such  tiling is \cite{gottschalk1944orbit} (see  \cite{bufetov2013limit, Solomyak1997, Robinson1996}).

These well-known fundamental facts (continuity of the action, compactness and minimality) are used routinely in tiling dynamics articles from the 90's to this day, for e.g., in \cite{SchmiedingTrevino2021, trevino2025quantitative, CortezDurand2008, Solomyak1997, bufetov2013limit, liu2026new, BaakeGaehlerGohlke2025, NagaiAkiyamaLee2022, SolomyakTrevino2024, KenyonSolomyak2010, GaehlerLiuSadun2026, Sadun2011, BargeDiamondHuntonSadun2010 , PerrymanSadun2025, WaltonWhittaker2024, BellissardBenedettiGambaudo2006, HuntonWalton2021, SmilanskySolomon2021, Kenyon1996, BellissardJulienSavinien2009, Julien2017, BaakeGrimm2020, BaakeGaehlerSadun2025 , BaakeGaehlerMazac2022, BaakeFrankGrimmRobinson2019 , BaakeGaehlerMazacMitchell2025 , MarshallMaldonado2026}. Therefore, we hope that this survey with full proofs provides a useful reference for experts. Moreover, the preliminary material on which these results depend (Section \ref{sec:tilings} and Section \ref{sec:hausdorff}), to the best of our knowledge, is not systematically treated in the literature, hence, the other purpose of this manuscript is to fill this expository gap by giving  self-contained proofs and by making explicit the underlying subtleties of these classical results, in order to be suitable for a broad audience, and in particular, to newcomers that are interested in the area of tiling dynamics and aperiodic order.

\section{Basic notation}
Throughout this manuscript,  $|| \cdot ||$ will denote the a norm on $\R^d$. For $r>0$, $B_r$ stands for the closed ball of radius $r$ centered at the origin. For $x\in\R^d$ and $\delta>0$, $B(x,\delta)$ denotes the open ball of center $x$ and radius $\delta$. $\inte{A}$ and $\overline{A}$ will denote the interior and the closure of a set $A$, respectively.  For $Y\subseteq\R^d$ and $z\in\R^d$, $d(z,Y)$ stands for the distance of the point $z$ to the set $Y$, that is $$d(z,Y):=\inf\{||z-y||:y\in Y\}.$$

For a Borel set $A\subseteq \R^d$, $\lambda(A)$ will denote its Lebesgue meausure.
\section{Tiles and Tilings}\label{sec:tilings}


A set $D \subseteq \mathbb{R}^{d}$ is called a \emph{tile} if it is compact and equal to the closure of its interior. We will always assume that tiles are connected, although in some situations it might be useful to allow disconnected tiles. Tiles in $\mathbb{R}$ are bounded closed intervals, tiles in $\mathbb{R}^2$ are often polygons, but fractal tiles also frequently appear in examples.

A tiling $x$ is a collection of tiles such that the interiors of the tiles are pairwise disjoint, and whose union covers $\mathbb{R}^d$. Note that a tiling has countable many tiles, because the interior of the tiles are non empty and pairwise, so for every tile of the tiling we can choose a point in its interior with rational coordinates.
Two tiles $D_{1}$ and $D_{2}$ are said to be equivalent, denoted by $D_1 \sim D_{2}$, if one is a translation of the other. The equivalence classes of this relation are called \emph{prototiles}. Two tiles are said to be \emph{congruent} if they are related by an orientation-preserving isometry (a rotation composed with a translation). We call the equivalence class of congruent tiles \emph{congruence prototiles}.

\begin{definition}
Let $\mathcal{T}$ be a finite collection of distinct prototiles (resp. congruence prototiles) in $\mathbb{R}^d$. We denote by $X_\mathcal{T}$ the set of all possible tilings of $\mathbb{R}^d$ whose tiles belong to some prototile in the collection $\mathcal{T}$. We refer to $X_\mathcal{T}$ as a \emph{full tiling space}.
\end{definition}

Note that given a tiling $x$ of the space $\mathbb{R}^d$, the translation of $x$ by an element $v \in \mathbb{R}^d$ is always another tiling. Once this is observed, we see that $\mathbb{R}^d$ induces a natural action on a fill tiling space made of prototiles.

\begin{definition}
For $v \in \mathbb{R}^d$ and $x \in X_\mathcal{T}$, let $T_{v}x$ be the tiling of $\mathbb{R}^d$ in which each tile $D \in x$ has been translated by $-v$. That is,
\[
T_{v}x: = \{D-v : D \in x\}.
\]
We denote this action of $\mathbb{R}^d$ on $X_\mathcal{T}$ by $T$. If $x'\subseteq x$ is a patch, then $x'-v$ stands for the patch $\{D-v : D \in x'\}$.
\end{definition}

\begin{definition}
Let $\mathcal{T}$ be a finite set of distinct prototiles (of congruence) in $\mathbb{R}^d$. A \emph{$\mathcal{T}$-patch} is a finite subset $y$ of some tiling $x \in X_{\mathcal{T}}$ such that the union of all tiles in $y$ is a connected set in $\mathbb{R}^d$. We call the union of the tiles of $y$ the support of the patch $y$, denoted by $supp(y)$.

We extend the notion of equivalence to patches; that is, two patches are equivalent if one is a translation of the other, and the corresponding equivalence classes are called \emph{protopatches}. We say that two protopatches are congruent if they are related by an orientation-preserving isometry, and the corresponding equivalence classes are referred to as \emph{congruence protopatches}.

If $\mathcal{T}$ consists of prototiles (resp. congruence prototiles), $\mathcal{T}^*$ denotes the set of all protopatches (resp. congruence protopatches), i.e, equivalence classes of $\mathcal{T}$-patches, and we use $\mathcal{T}^{(n)}$ to denote the collection of protopatches  (resp. congruence prototiles) of $\mathcal{T}^*$ such that the representative of protopatch consists of $n$ tiles, that $$\mathcal{T}^{(n)}:=\{[x']\in \mathcal{T}^*:x' \text{ has }n  \text{ tiles}\}.$$

We will abuse of terminology and refer to protopatches as if they were patches, hence, by a protopatch of $n$ tiles we refer to a protopatch whose representative has $n$ tiles, and sometimes, as long as there will not be any confusion, we will consider $\mathcal{T}^*$ to be a set that contains one and only one representative of each protopatch (resp. congruence protopatch).
\end{definition}

Let $K \subset \R^d$ be compact connected and let $x\in X_{\mathcal{T}}$. We denote by $x[[K]]$ the collection of patches $x'$ contained in $x$ with the property that $\mathrm{supp}(x') \supseteq K$. We denote $x(K):= \{D\in x: D\cap K\neq\emptyset\}$.

\begin{proposition}\label{prop:parche_de_intersecciones}
    Let $x\in X_{\mathcal{T}}$ be a tiling and let $K\subseteq \R^d$ be a compact connected set. Then $x(K)\in x[[K]]$, i.e, $x(K)$ is a patch.
\end{proposition}

\begin{proof}
    Since every element of $x(K)$ is a connected set that intersects the connected set $K$, then $K\cup\displaystyle\bigcup_{D\in x(K)}D=K\cup\operatorname{supp}(x(K))$ is connected. But $K\cup\operatorname{supp}(x(K))= \operatorname{supp}(x(K))$, so $\operatorname{supp}(x(K))$ is connected. To see that $x(K)$ is finite, let $r>0$ such that $K\subseteq B_r$. Let $\delta>0$ such that every tile contains a ball of radius $\delta$ in its interior, and consider 

    \begin{equation*}
        R:=2\max\{\operatorname{diam}(T): T \in\mathcal{T}\}.
    \end{equation*}

    It follows that every tile in $x(K)$ must be contained in $B_{r+R}$. Therefore, using that every tile contains a ball of radius $\delta$ in its interior, together with the fact that the interiors of the tiles are pairwise disjoint, we obtain the following inequality regarding the cardinality of $x(K)$

    \begin{equation*}
        0<|x(K)|\lambda(B_\delta)\leq \lambda(B_{r+R})<\infty,
    \end{equation*}

     so $x(K)$ is finite.
\end{proof}

\begin{remark}
    Observe that if $K$ is not connected in Proposition \ref{prop:parche_de_intersecciones}, then $x(K)$ might not be patch; just consider $K$ to be a set of two points where each point lies at the interior of a tile in $x$ where both respective tiles are not adjacent to each other.
\end{remark}

Note that for an arbitrary compact connected set $K$, there might be no patch in $x[[K]]$ that is the smallest in the sense of inclusion, to see this, just consider a point that lies in the boundary of two tiles in $x$. Nevertheless, if $K$ is also a tile set (not necessarily in $\mathcal{T}$), then there is such a patch.

\begin{proposition}\label{prop:tile_interior}
    Let $x\in X_{\mathcal{T}}$ and $K\subseteq\R^d$ be a compact connected set equal to the closure of its interior, i.e, a tile. For every point in $K$ there is a tile $D$ in $X_{\mathcal{T}}$ such that $x\in D$ and $\inte{D}\cap\inte{K}\neq\emptyset$.
\end{proposition}
\begin{proof}
    Consider $$\mathcal{A}:=\{A\in x : \inte{A} \cap\inte{K}\neq\emptyset\}.$$First we prove that $\mathcal{A}$ is not empty.  Indeed, let $v\in\inte{K}$, let $A\in x$ such that $v\in A$ and let $\delta>0$ such that the ball $B(v,\delta)$ is contained in $\inte{K}$. Since $A$ is a tile, $B(v,\delta)\cap \inte{A}\neq\emptyset$, so $\inte{A}\cap\inte{K}\neq\emptyset$, and $A\in\mathcal{A}$, therefore $\mathcal{A}\neq\emptyset$. Furthermore, every point in $\inte{K}$ lies in a tile in $\mathcal{A}$.  
    
    Now, let $w\in K$. Since $K$ is equal to the closure of its interior, there is a sequence $(w_n)_{n\in\N}$ in $\inte{K}$ such that $w_n \to w$. It is clear from the previous paragraph that for every $n\in\N$ we have $ w_n\in\displaystyle\bigcup_{A\in\mathcal{A}}A$, and observe that $A$ is finite because it is a subset of $x(K)$. Thence, $w$ is in the closed set $\displaystyle\bigcup_{A\in\mathcal{A}}A$. We deduce that there is a tile $D\in\mathcal{A}$ such that $w\in D$.
\end{proof}

For a compact connected set $K\subseteq\R^d$ that is equal to the closure of its interior we define $$x[K]:=\{A\in x:\inte{A}\cap\inte{K}\neq\emptyset\}.$$

\begin{proposition}\label{prop:smallest_patch}
     Let $x\in X_{\mathcal{T}}$ and $K\subseteq\R^d$ be a compact connected set equal to the closure of its interior, i.e, a tile. $x[K]$ is the smallest patch in $x[[K]]$.
\end{proposition}

\begin{proof}
     By Proposition \ref{prop:tile_interior}, $x[K]$ covers $K$. Since every element of $x[K]$ is a connected set that intersects the connected set $K$, then $K\cup\displaystyle\bigcup_{D\in x[K]}D=K\cup\operatorname{supp}(x[K])$ is connected. But $K\cup\operatorname{supp}(x[K])= \operatorname{supp}(x[K])$, so $\operatorname{supp}(x[K])$ is connected. $x[K]\subseteq x(K)$, thus $x[K]$ is finite. To see that is the smallest, let $x'\in x[[K]]$, let $D\in x[K]$, and let $d\in \inte{D}\cap\inte{K}$. Since $d\in K$, there $A\in x'$ such that $d\in A$, therefore $A\cap\inte{D}\neq\emptyset$, hence $A=D$ and $D\in x'$, thus $x[K]\subseteq x'$. 
\end{proof}

Observe that the intersections of two patches of a tiling might not be a patch, since the union of the tiles that are in the intersection of two patches need not be connected. However, if two patches $x',y'$ of a same tiling $x$ cover a same tile set $K$, then $x' \cap y'$ contains a subset that is a patch that covers $K$.

\begin{proposition}\label{Prop:interseccion_de_parches}
    Let $x\in X_{\mathcal{T}}$ be a tiling, let $K\subset \R^d$ a compact connected set equal to the closure of its interior, i.e, a tile set. Let $x',y'\in x[[K]]$. Then, there there is $z'\subseteq x'\cap y'$ such that $z'\in x[[K]]$.
\end{proposition}
\begin{proof}
   By Proposition \ref{prop:smallest_patch}, $x[K]\subseteq x'\cap y'$.
\end{proof}
\begin{definition}
We say that a full tiling space has \emph{finite local complexity (f.l.c) under translation} if $\mathcal{T}$ consists of a finite number of prototiles and $\mathcal{T}^{(2)}$ is finite. That is, there are only finitely many adjacent pairs of tiles.
\end{definition}

\begin{definition}
A full tiling space $X_\mathcal{T}$ has \emph{finite local complexity under the Euclidean group} if $\mathcal{T}$ consists of a finite number of congruence prototiles and $\mathcal{T}^{(2)}$ is finite.
\end{definition}

Sometimes the geometry of the tiles themselves imposes this condition of finite local complexity under translation, but usually it must be added as an extra assumption. When working with polygonal tiles, a typical way to achieve finite local complexity under translation is to add ``tabs'' that force the tiles to fit together in only finitely many ways, like pieces of a puzzle. Another way to achieve finite local complexity is to allow tiles to meet only \emph{edge-to-edge}.

Intuitively, since every protopatch formed by three tiles or more is constructed by attaching a prototile to a protopatch in  $\mathcal{T}^{(2)}$, f.l.c will guarantee that $\mathcal{T}^{(n)}$ is finite for every $n$. However, we will provide the argument that justifies this idea.

\begin{proposition}\label{prop:patches_conectados}
    Let $p$ be a patch with more than one tile, then there is a tile $D\in p$ such that $p\setminus\{D\}$ is still a patch.
\end{proposition}

\begin{proof}
    All we need to prove is that there is $D\in p$ such that $\displaystyle\bigcup_{A\in p\setminus\{D\} }A$ is still connected.

    Consider the graph $G_p$ where the vertices are the tiles of $p$ and two tiles of $D_1,D_2 \in p$ are connected by an edge if and only if $D_1 \cap D_2 \neq\emptyset$. We claim that $G_p$ is connected. To see this, suppose that $G_p$ is not connected. Let $v_{1},...,v_n$ be the set of vertices of each of its connected components. Then, it follows that for every $i\in\{1,...,n\}$, $v_i$ is a patch and $\operatorname{supp}(v_i)$ is closed. Consequently, $\operatorname{supp}(v_1),...,\operatorname{supp}(v_n)$ are closed sets pairwise disjoint, therefore, $\displaystyle\bigcup_{i=1}^{n}\operatorname{supp}(v_i)=\operatorname{supp}(p)$ is disconnected, which is a contradiction.

    Since $G_p$ is a connected graph, we have that there is a vertex $D$ such that once we remove it, the remaining graph is still connected. Let $G$ be the graph obtained after removing the vertex $D$ of $G_p$, i.e the graph induced by the subset of vertices $p\setminus\{D\}$. Since $G$ is connected, there is walk in $G$ that contains all the vertices, i.e, there is $l\in\N$ and $K_1,...,K_l\in p\setminus\{D\}$ such that for all $i\in\{1,...,l-1\}$ we have $K_i\cap K_{i+1}\neq\emptyset$ and $\displaystyle\bigcup_{i=1}^{l}K_i=\displaystyle\bigcup_{A\in p\setminus\{D\}}A$. By the connectedness of each $K_i$, we deduce that $\displaystyle\bigcup_{i=1}^{l}K_i$ is connected. Hence $p\setminus\{D\}$ is a patch.
\end{proof}

\begin{proposition}\label{prop: flc}
    A full tiling space has finite local complexity if and only if $\mathcal{T}^{(n)}$ is finite for every $n\in\N$.
\end{proposition} 
\begin{proof}
    By induction, assume that $\mathcal{T}^{(n)}$ is finite for $n\geq2$. Observe that, by Proposition \ref{prop:patches_conectados}, the elements of $\mathcal{T}^{(n+1)}$ may be viewed as the elements of $\mathcal{T}^{(n)}$ where some of its tiles is attached to another tile. Since $\mathcal{T}$, $\mathcal{T}^{(2)}$ and $\mathcal{T}^{(n)}$ are finite, there is only finitely many ways to produce an element of $\mathcal{T}^{(n+1)}$ from an element of $\mathcal{T}^{(n)}$. That is, $\mathcal{T}^{(n+1)}$ is finite.

    Formally, let us suppose that $\mathcal{T}^{(n+1)}$ is infinite. Let $\{x_k:k\in\N\}$ be a infinite countable set of pairwise non-equivalent patches of $n+1$ tiles such that $\{[x_k]:k\in\N\}\subseteq \mathcal{T}^{(n+1)}$, let us denot $\mathcal{P}:=\{x_k:\in\N\}$. For each $k\in\N$ let us consider $D_{k} \in x_k$ such that $x_k\setminus\{D_k\}$ is a patch so that $[x_k\setminus\{D_k\}]\in \mathcal{T}^{(n)}$ (by Proposition \ref{prop:patches_conectados}). Consider the map $f\colon \mathcal{P}\to \mathcal{T}^{(n)}$ given by $f([x_k]):=[x_k\setminus\{D_k\}]$. Since $\mathcal{P}$ is infinite, there is $[y]\in \mathcal{T}^{(n)} $ such that $f^{-1}(\{[y]\})$ contains a countable infinite subset, hence $\mathcal{P}\subseteq f^{-1}(\{[y]\})$ (up to a subsequence). For every $k\in\N$  $((x_k\setminus\{D_k\})\cup D_k$ produces at least one patch of two tiles with $D_k$ as one of its tiles. Choose such a patch and denote it by $\{D_k ,A_k\}$. Consequently, for every $k\in\N$, $ A_k\in x_k\setminus\{D_k\}$. Since $\mathcal{T}$ is finite, we have that (up to a subsequence) $$\forall i,j\in\N,A_i\sim A_j.$$Since $\mathcal{T}^{(2)}$ is finite, it follows that for every $i,j\in \N$ (up to a subsequence) 
    \begin{equation*}
        \{D_i ,A_i\} \text{ is equivalent to } \{D_j ,A_j\}.
    \end{equation*}
     Since for every $i,j\in\N$ we have that $x_i\setminus\{D_i\}$ is a translation of $x_j\setminus\{D_j\}$, and since $\mathcal{P}$ is infinite, it necessary that for any fix $i\in \N$, $x_i\setminus\{D_i\}$ contains infinitely many translations of $A_i$, which is a contradiction because $x_i\setminus\{D_i\}$ is a patch (it has finitely many tiles). We conclude that $\mathcal{T}^{(n+1)}$ is finite.
\end{proof}

\section{Hausdorff distance}\label{sec:hausdorff}
Since we will be working with tiles and patches, which are sets, it will be useful to have a notion of convergence of sets. Thus, we recall the Hausdorff distance.

\begin{definition}
    Let $A,B\subseteq \R^d$ and consider a norm $||\cdot||$ in $\R^d$. The Hausdorff distance between $A$ and $B$ is defined as

\begin{equation*}
    h(A,B):=\max\{\sup_{a\in A}d(a,B), \sup_{b\in B}d(b,A) \}.
\end{equation*}
\end{definition}

We recall that $h(\cdot,\cdot)$ is a metric on $\{X\subseteq \R^d: X \text{ is compact}\}$, and we will enunciate some propositions regarding the Hausdorff distance that will be used in the sequel. 

\begin{proposition}\label{prop:distancia_traslaciones}
    Let $A\subseteq \R^d$ be compact and let $t\in \R^d$. Then $$h(A,A+t)=||t||.$$
\end{proposition}

\begin{proof}
    The case $t=0$ is trivial so suppose $t\neq 0$. We begin by showing that $\displaystyle\sup_{a\in A}d(a,A+t)=||t||$. First, observe that for every $a\in A$ we have $$d(a,A+t)\leq ||a-(a+t)||=||t||,$$ so $\displaystyle\sup_{a\in A}d(a,A+t)\leq||t||$. To see that  $\displaystyle\sup_{a\in A}d(a,A+t)\geq||t||$, consider $\bar{a}\in A$ such that 
\begin{equation*}
    \inner{\bar{a}}{\frac{t}{||t||}}=\min_{x\in A}\inner{x}{\frac{t}{||t||}}.
\end{equation*}

If $a\in A$, we obtain that

\begin{equation*}
    \begin{aligned}
        ||\bar{a}-(a+t)||&\geq|\inner{\bar{a}-(a+t)}{\frac{t}{||t||}}|\\
        &=|\inner{\bar{a}}{\frac{t}{||t||}}-\inner{a}{\frac{t}{||t||}}-||t|||\\
        &=||t||+\inner{a}{\frac{t}{||t||}}-\inner{\bar{a}}{\frac{t}{||t||}}\\
        &\geq||t||.
    \end{aligned}
\end{equation*}

Thus, $d(\bar{a},A+t)\geq ||t||$, therefore, $\displaystyle\sup_{a\in A}d(a,A+t)\geq||t||$ and $$\displaystyle\sup_{a\in A}d(a,A+t)=||t||.$$

Analogously, but using $\bar{a}\in A$ such that $$\inner{\bar{a}}{\frac{t}{||t||}}=\displaystyle\max_{x\in A}\inner{x}{\frac{t}{||t||}},$$
we see that $\displaystyle\sup_{a\in A+t}d(a,A)=||t||$. Therefore $h(A,A+t)=||t||$.
\end{proof}

\begin{corollary}\label{cor:convergencia_traslaciones}
    Let $A\subseteq\R^d$ be a compact set, $t\in\R$ and $(t_n)_{n\in\N}$ a sequence in $\R^d$. Then $A +t_n$ converges to $A+t$ with the Hausdorff distance if and only if $t_n \to t$.
\end{corollary}

\begin{proposition}\label{prop:convergencia_rot}
    Let $A\subseteq\R^d$ be a compact set, $M\in SO(d)$ and $(M_n)_{n\in\N}$ a sequence in $SO(d)$ such that $M_n\to M$. Then, $M_n (A)$ converges to $M(A)$ with the Hausdorff distance.
    
\end{proposition}

\begin{proof}
     Since $A$ is compact, let $C\in\R$ such that for every $a\in A, ||a||\leq C$. Let $x\in A$, $d(M_nx,M(A))\leq ||M_nx-Mx||\leq ||M_n-M||||x||\leq C||M_n-M||$ (observe that we are using the same notation for the norm in $\R^d$ and the norm for matrices). Analogously, we have that $d(Mx,M_n(A))\leq C||M_n-M||$. Therefore, $h(M_n(A),M(A))\leq C||M_n-M||,$ and $M_n(A)\to M(A)$ with the Hausdorff distanice.
\end{proof}

\begin{proposition}\label{prop:contencion_convergencia hausdorff}
    Let $A_n,B_n\subseteq \R^d, n\in\N$, be compact sets. Let $A,B\subseteq\R^d$ be compact sets such that $A_n\to A$ and $B_n\to B$. Suppose that for every $n\in\N, A_n\subseteq B_n$. Then $A\subseteq B$.
\end{proposition}

\begin{proof}
    Let $a\in A$. Since $A_n\to A$, we have that $d(a,A_n)\to 0$. Hence, using that $A_n \subseteq B_n$, we deduce that $d(a,B_n)\to 0$. Let $\varepsilon>0$, because $B_n \to B$, we can choose $n_0 \in\N$ such that for every $n\geq n_0$

    \begin{equation*}
        d(a,B_n)\leq\frac{\varepsilon}{2} \text{ and } \sup_{b\in B_n} d(b,B)\leq \frac{\varepsilon}{2}.
    \end{equation*}

    Consequently, $d(a,B)\le d(a,B_n)+ \displaystyle\sup_{b\in B_n} d(b,B)\leq \varepsilon$.

    Since $\varepsilon>0$ was arbitrary, it follows that $d(a,B)=0$. Since $B$ is closed, $a\in B$.
\end{proof}


\section{Tiling topology}\label{sec:tiling_topology}

\begin{definition}

We define the metric $d_1$ on a full tiling space $X_\mathcal{T}$ made up from prototiles as

\begin{equation}\label{def_metrica}
    d_1 (x,y) = \min\left(\dfrac{1}{\sqrt{2}},\tilde{d}_1(x,y)\right), 
\end{equation}

where $\tilde{d}_1(x,y)$ is defined as
$$\inf\left\{ 0 < r < \dfrac{1}{\sqrt{2}} : \exists v _1 ,v_2\in B_r, \text{ such that } T_{v_1} x[B_{\frac{1}{r}}] = T_{v_2}y[B_\frac{1}{r}] \right\}.$$
\end{definition}
  Observe that $T_{v_1}x[B_\frac{1}{r}]=x[B_\frac{1}{r}+v_1]-v_1$. The idea is that two tilings are close if, after a small translation, they agree on a large ball around the origin. We refer to $d_1$ as the tiling metric

We hereby demonstrate that $d_1$ is a metric.

\begin{proposition}
    $d_1$ is a metric on $X_\mathcal{T}$.
\end{proposition}

\begin{proof}
    Let $x,y\in X_\mathcal{T}$ such that $d_1(x,y)=0$, we will prove that $x=y$. Let $D\in x$ be a tile, since $D$ is compact, consider $r>0$ such that $D\subseteq B_{r}$. By definition of  $d_1(x,y)=0$, we can consider a sequence of positive numbers $(r_n)_{n\in\N}$ with $r_n\to0$ such that 
    \begin{equation*}
        \forall n\in\N, \exists v_{n},u_{n}\in B_{r_n}, \text{ such that } T_{v_{n}} x[B_{\frac{1}{r_n}}] = T_{u_n}y[B_\frac{1}{r_n}] 
    \end{equation*}

    Since $r_n\to 0$, then there is $n_0$ such that for every $n\geq n_0$  we will have $\frac{1}{r_n}-r_n>r$, so that $D-v_n \in T_{v_n}x[B_{\frac{1}{r_n}}]$. Thus $D+u_n-v_n \in y[B_\frac{1}{r_n}+u_n]\subseteq y$ for every $n\geq n_0$. Let $d\in \operatorname{int}(D)$ and $\delta>0$ such that $B(d,\delta)\subseteq \operatorname{int}(D)$. Consider $m_0\geq n_0$ such that if $n\geq m_0$, then $r_n<\frac{\delta}{4}$. It follows that $$\forall n,m\geq m_0,\;d\in \operatorname{int}(D+u_n-v_{n})\cap\operatorname{int}(D+u_m-v_{m}).$$ But since $y$ is a tiling, the interiors of the tiles of $y$ are pairwise disjoint, consequently, for every $n,m\geq m_0$ we have that $u_n -v_n =u_m-v_m$. Therefore, $u_n-v_n$ is an eventually constant sequence that converges to $0$, so it is eventually $0$. As a result, $D\in y$ and $x\subseteq y$. Interchaging the roles of $x$ and $y$ we see that $y\subseteq x$ and we conclude that $x=y$.\\

    For the triangle inequality: let $0 < d_1(x,y) = a'$ and $0 < d_1(y,z) = b'$ with $a'+b' < 1/ \sqrt{2}$ (the non-trivial case). Let $0< \epsilon < 1/ \sqrt{2} - (a'+ b')$ and define $a := a' + \epsilon/2$, $b := b' + \epsilon/2$.\\
    
    By definition of $d_1$, there is $v_1,v_2,v'_2, v'_3\in\R^d$ with $||v_1||,||v_2||\leq a$ and $||v'_2||,||v'_3||\leq b$ such that
    \begin{align*}
        T_{v_1}x[ B_{\frac{1}{a}}]
        &=
         T_{v_2}y[ B_{\frac{1}{a}}].
        \\
         T_{v'_2}y[ B_{\frac{1}{b}}]
        &=
       T_{v'_3}z[ B_{\frac{1}{b}}]
         .
    \end{align*}

Consequently,
$$
T_{v_1+v'_2}x[ B_{\frac{1}{a}}-v'_2]
=T_{v_2+v'_2}y[ B_{\frac{1}{a}}-v'_2].
$$

Since
$$B_{\frac{1}{a}-b}\subseteq B_{\frac{1}{a}}-v'_2,
$$
It follows that
\begin{equation}\label{desig_triangular}
T_{v_1+v'_2}x[ B_{\frac{1}{a}-b}]
=T_{v_2+v'_2}y[ B_{\frac{1}{a}-b}].
\end{equation}

Analogously,
\begin{equation}\label{desig_triangular_2}
T_{v_2+v'_2}y[ B_{\frac{1}{b}-a}]=T_{v_2+v'_3}z[ B_{\frac{1}{b}-a}].
\end{equation}

Because of $a,b<\frac{1}{\sqrt{2}}$, we have that 
$$0<\dfrac{1}{a+b} \leq \min\{\dfrac{1}{b} - a,\dfrac{1}{a} - b\}.$$
Hence, combining \ref{desig_triangular} and \ref{desig_triangular_2} we obtain that
$$
T_{v_1+v'_2}x[ B_{\frac{1}{a+b}}]=T_{v_2+v'_3}z[ B_{\frac{1}{a+b}}].
$$

   Using that $||v_1+v'_2||,||v_2 +v'_3||\leq a+b$, we get $d_1(x,z) \leq a+b = d_1(x,y) + d_1(y,z) + \epsilon$; since $\epsilon\in(0, 1/ \sqrt{2} - (a'+ b'))$ was arbitrary, $d_1$ satisfies the triangle inequality.
\end{proof}
\begin{remark}
    The $\frac{1}{\sqrt{2}}$ bound in Definition \ref{def_metrica} is just to obtain the triangle inequality.
\end{remark}

    The following Theorem synthesizes the role of finite local complexity in the full tiling.
\begin{theorem} \label{theo: compacidad d1}
    $X_\mathcal{T}$ has finite local complexity under translation if and only if $X_\mathcal{T}$ with $d_1$ is a compact space.
\end{theorem}

In order to prove this, we require the following lemma.

\begin{lemma}[Selection Lemma]

Given a sequence $D_{n}$ of tiles, all sharing a common point $p$ and each congruent (resp. equivalent) to a same tile $D'$, there exists a subsequence $D_{n_k}$ converging to a tile congruent (resp. equivalent) to $D'$ such that $p\in D'$.
    
\end{lemma}

\begin{proof}
    Let $d_n \in D'$ be the point corresponding to $p$ for each tile of the sequence. By compactness of $D'$, there exists a subsequence $d_{n_k}$ converging to some $\bar{d}\in D'$. We translate $D'$ so that $\bar{d}$ coincides with $0\in\R^d$; that is, each tile $D_{n_k}$ is obtained by rotating $D'$ by a rotation matrix in $SO(d)$ 
    and then translating it by a vector $t_{n_k}$.\\

    Since all tiles of the subsequence contain $p$, it follows that the sequence $t_{n_k}$ is bounded, so it has a subsequence converging to some $t\in \R^d$, which we continue to denote $t_{n_k}$. Since $SO(d)$ is compact, the sequence $A_{n_k}$ of rotation matrices corresponding to each $D_{n_k}$ has a subsequence converging to some $A\in SO(d)$, which we continue to denote $A_{n_k}$.\\

    This implies that the subsequence $D_{n_k}$ converges to a tile $D$ obtained by rotating $D'$ by a matrix $A\in SO(d)$ around $\bar{d}$ and translating by a vector $t$ (Proposition \ref{prop:convergencia_rot} and Corollary \ref{cor:convergencia_traslaciones}), which is congruent to $D'$. Note that $D$ contains $p$, since the sequence that rotates $d_{n_k}$ by  $A_{n_k}$ and translates it by $t_{n_k}$ is constantly equal to $p$, and converges to $\bar{d}$ translated by $t$ (since the rotation is around $\bar{d}$).
\end{proof}

    Let us prove Theorem \ref{theo: compacidad d1}.
    
    \begin{proof}
   Assume that a full tiling space $X_\mathcal{T}$ has finite local complexity under translation. Let $(x_i)_{i\geq 1}$ be sequence of tilings. Consider

   \begin{equation*}
       R:=2\max\{\operatorname{diam}(T): T \in\mathcal{T}\}.
   \end{equation*}

    Now we will inductively construct a convergent subsequence. where the induction  over $n\geq 1$ is such that:

    \begin{enumerate}
        \item $(x_{i}^{n})_{i\geq1} $ is a subsequence of $(x_{i}^{n-1})_{i\geq1}$, where we define $x_{i}^{0} := x_{i}$.
        \item for each $i,j\geq 1$ we have $x_{i}^{n}[B_{nR}]\sim x_{j}^{n}[B_{nR}]$, and the sequence $(\operatorname{supp}(x_{i}^{n}[B_{nR}]))_{i\geq 1}$ converges with the Hausdorff metric to the support of $p^n$, where $p_n$ is a patch that covers $B_{nR}$ and belongs to the same protopatch as the sequence.
        \item $p^{n-1}\subseteq p^n$, where $p^0:=\emptyset$. Moreover, $\operatorname{supp}(p^n) \subseteq B_{(n+1)R}$. 
    \end{enumerate}

    We start with $n=1$.  For each $i\geq 1$, let $q_{i}^{1} = x_{i}[B_R]$. Observe that the sequence of the cardinalities of the patches $|q_{i}^{1}|$ is bounded above, because if $T\in q_{i}^{1}$ for some $i$, it follows that $T\subseteq B_{2R}$. Hence
    
    \begin{equation*}
            \forall i\geq 1, \operatorname{supp}(q_{i}^{1})\subseteq B_{2R}.
    \end{equation*}

   On the other hand, since there is a finite number of prototiles, and, by definition of tile, every tile has nonempty interior, there is $r>0$ such that every $T\in\mathcal{T}$ contains a ball of radius $r$ on its interior. Thus, it follows that
    \begin{equation*}
        \forall i\geq 1, \lambda(B_r)|q_{i}^{1}|\leq\lambda(B_{2R}).
    \end{equation*}
     Consequently, the sequence $|q_{i}^{1}|$ is bounded above.

     Thence, there exists a subsequence of $|q_{i}^{1}|$ that is constant. We continue to denote this subsequence by $|q_{i}^{1}|$ to avoid overburdening the notation.

    Therefore, $q_{i}^{1}$ is a sequence of patches with the same number of tiles such that every patch of the sequence covers $B_{R}$. Say they have $k$ tiles. Since we have finite local complexity under translation, we know that $\mathcal{T}^{(k)}$ is finite. It follows that there exists a subsequence $q_{\phi(i)}^{1}$ such that every patch of the subsequence belongs to the same protopatch $[q'] \in \mathcal{T}^{(k)}$. We now have a sequence $q_{\phi(i)}^{1}$ of patches, each covering $B_{R}$ with its support contained in $B_{2R}$; then, by an argument identical to that of the selection Lemma (applied to the sequence of the supports of the patches), up to passing to a subsequence, $\operatorname{supp}(q_{\phi(i)}^{1})$ converges to a set that is the support of a patch $q$ that belongs to the protopatch $[q']$, such that $q$ also covers $B_R$ and its support is also contained in $B_{2R}$ (by Proposition \ref{prop:contencion_convergencia hausdorff}). So choose 
    \begin{equation*}
        x^{1}_i:=x_{\phi(i)},\;  p^1 :=q.
    \end{equation*}

    We now assume the inductive hypotheses for $n-1$, with $n\geq2$. For each $i\geq 1$, let $q_{i}^{n} := x^{n-1}_{i}[B_{nR}]$. Again, the sequence $|q_{i}^{n}|$ is bounded above (same reason as the base case), therefore it follows that there exists a constant subsequence of $|q_{i}^{n}|$. So, as the base case, there is a subsequence $q_{\phi(i)}^{n}$ such that each term belongs to the same protopatch. We now have a sequence $q_{\phi(i)}^{n}$ of patches, each covering $B_{nR}$ with its support contained in $B_{(n+1)R}$. Then, by an argument identical to that of the selection Lemma (applied to the sequence of the supports), after passing to a subsequence, $\operatorname{supp}(q_{\phi(i)}^{n})$ converges to a set that is the support of a patch $q^n$ that belongs to the same protopatch as the sequence, with $q^n$ also covering $B_{nR}$ and with its support also contained in $B_{(n+1)R}$ (by Proposition \ref{prop:contencion_convergencia hausdorff}), so it suffices to put
    \begin{equation*}
        x^n_{i} := x^{n-1}_{\phi(i)},\; p^n:= q^n.
    \end{equation*}

    Consequently, condition $(1)$, $(2)$ and $(3)$ are satisfied.

    Now we will define a subsequence diagonally. Let $(y_n)_{n\geq 1}$ be the subsequence of $(x_n)_{n\geq 1}$ given by $y_n := x_n^{n}$ and  consider $y:=\displaystyle\bigcup_{n\in\N} p_n$.  Since $(p_n)_{n\geq1}$ is an increasing sequence of patches, it is clear that the interior of the tiles in $y$ are pairwise disjoint, and using that $$\R^d=\bigcup_{n\geq1}B_{nR}\subseteq\bigcup_{n\geq 1}\operatorname{supp}(p_n),$$ it follows that $y\in X_\mathcal{T}$, i.e, $y$ is a tiling.

    We will finish by showing that $y_n\to y$. Let $\varepsilon>0$. Consider $j_0\in\N$ such that $j_0 R \geq \frac{1}{\varepsilon}$. Using $(2)$ and Corollary \ref{cor:convergencia_traslaciones}, let $i_0 \in\N$ be such that if $i\geq i_0$ then there is $t_i \in\R^d$ with $||t_i||\leq\varepsilon$ such that $x_{i}^{j_0}[B_{j_0R}]+t_i=p^{j_0}$. Hence, by the fact that $j_0 R\geq \frac{1}{\epsilon}$, it follows that $$x_{i}^{j_0}[B_{\frac{1}{\epsilon}}]+t_i\subseteq p^{j_0}.$$
    
    Since for $n\geq j_0$ we have that $(x_{i}^{n})_{i\geq 1}$ is a subsequence of $(x_{i}^{j_0})_{i\geq 1}$, we deduce that $$\forall i \geq i_0, \forall n\geq j_0, \; x_{i}^{n}[B_{\frac{1}{\epsilon}}]+t_i\subseteq p^{j_0}.$$ Defining $n_0 :=\max\{i_0 , j_0\}$, we obtain that $$\forall n\geq n_0, \; x_{n}^{n}[B_{\frac{1}{\epsilon}}]+t_n\subseteq p^{j_0}\subseteq y.$$
    Consequently, $$\forall n\geq n_0, \; x_{n}^{n}[B_{\frac{1}{\epsilon}}]\subseteq T_{t_n}y.$$ Therefore $$\forall n\geq n_0,\; x_{n}^{n}[B_{\frac{1}{\epsilon}}]=T_{t_n}y[B_{\frac{1}{\epsilon}}].$$Thus, $$n\geq n_0,\; d_1(y_n,y)\leq \varepsilon.$$ Hence, compactness is concluded.

    For the reverse implication, assume that $d_1$ makes $X_\mathcal{T}$ a compact metric space. Consider $r \in (0,\frac{\sqrt{2}}{2})$ such that $R<\frac{1}{r}-r$, and by compacntess, consider $x_1 ,...,x_n \in X_\mathcal{T}$ such that 
    \begin{equation}\label{theo:comp_d1_uso_compacidad}
          X_\mathcal{T} = \bigcup_{i=1}^{n}B(x_i,r).
    \end{equation} 

   Now, consider a patch $x'$ of two tiles. The patch $x'$ belongs to a tiling $x\in X_{\mathcal{T}}$. Let $v$ be a point of $\R^d$ that lies in the intersections of the two tiles of $x'$. Consider the tiling $T_vx$ so that the patch $x'-v$, which belongs to the same protopatch as $x'$, appears in $T_vx$. Observe that $\operatorname{supp}(x' -v) \subseteq B_R$. Using \eqref{theo:comp_d1_uso_compacidad}, we obtain that 

   \begin{equation*}
       \exists i\in\{1,...,n\}, \exists v_1,v_2\in B_r \text{ such that } T_{v_1}(T_vx)[B_{\frac{1}{r}}]=T_{v_2}x_i[B_\frac{1}{r}].
   \end{equation*}

   Since $R<\frac{1}{r}-r$, we have that $B_R \subseteq B_{\frac{1}{r}}+v_1$, hence 
   
   \begin{align*}
       x'-v-v_1\in T_{v_1}(T_vx)[B_{\frac{1}{r}}]&=T_{v_2}x_i[B_\frac{1}{r}]\\
       &=x_i[B_\frac{1}{r}+v_2]-v_2\\&\subseteq x_i[B_{\frac{1}{r}+r}]-v_2
   \end{align*}

   Thence, $x'-v-v_1+v_2\in x_i[B_{\frac{1}{r}+r}]$ and $x'$ is equivalent to a patch of $x_i[B_{\frac{1}{r}+r}]$. That is, since $x'$ was arbitrary, every patch of two tiles is equivalent to some patch of size two whose tiles are in $x_i[B_{\frac{1}{r}+r}]$ for some $i\in\{1,...,n\}$ (finitely many possibilities), thus $\mathcal{T}^{(2)}$ is finite. 



\end{proof}

\begin{proposition}\label{prop:continuidad_accion}

Let $X_{\mathcal{T}}$ be a full tiling space. The action by translation is continuous, that is, the map $T\colon\R^d\times X_{\mathcal{T}}\to X_\mathcal{T} $ given by $T(v,x):=T_vx$ is continuous.
    
\end{proposition}

\begin{proof}
     Let $t\in\R^d$, let $(t_n)_{n\in\N}$ a sequence in $\R^d$ such that $t_n\to t$ and let $(x_n)_{n\in\N}$ be a sequence in $X_\mathcal{T}$ such that $x_n$ converges to $x\in X_{\mathcal{T}}$. Let $\varepsilon>0$, assume that $\varepsilon< \frac{\sqrt{2}}{2}$. We need to find $n_0\in\N$ such that
     \begin{equation}\label{convergencia_tilings_trasladados}
       \forall n\geq n_0, d_1(T_{t_n}x_n, T_tx)<\varepsilon.
   \end{equation}

   consider $m_0$ such that 
   \begin{equation*}
    \forall n\geq m_0, ||t_n-t||< \frac{\varepsilon}{2}\leq\frac{1}{\varepsilon}. 
    \end{equation*}
   On the other hand, consider $k_0\in\N$ such that
    \begin{equation}\label{convergencia_tilings}
         \forall n\geq k_0, d_1(x_n,x)\leq\frac{1}{\frac{2}{\varepsilon}+||t||}.
    \end{equation}

   It is enough to consider $n_0:=\max\{k_0,m_0\}$. Indeed, let $n\geq n_0$, hence, by \eqref{convergencia_tilings}, there are $v_n,u_n\in B_\frac{1}{\frac{2}{\varepsilon}+||t||}$ such that 
   \begin{equation}\label{Consecuencia_convergencia}
        T_{v_n} x_n[B_{\frac{2}{\varepsilon}+||t||}]= T_{u_n}x[B_{\frac{2}{\varepsilon}+||t||}].
   \end{equation}

   To obtain \eqref{convergencia_tilings_trasladados}, we need $v'_n,u'_n\in B_{\epsilon}$ such that $$T_{v'_n+t_n}x_n[B_{\frac{1}{\epsilon}}]=T_{u'_n+t}x[B_{\frac{1}{\epsilon}}].$$

   Let 
    \begin{equation*}
        v'_n := v_n +t-t_n,\; u'_n:= u_n.
    \end{equation*}
    Observe that $$ T_{v_n} x_n[B_{\frac{2}{\varepsilon}+||t||}]-t\subseteq T_{v'_n}T_{t_n}x_n=T_{v_n+t}x_n,$$ and $$B_{\frac{1}{\epsilon}}\subseteq \operatorname{supp}( T_{v_n} x_n[B_{\frac{2}{\varepsilon}+||t||}]-t).$$ On the other hand, $$T_{u_n}x[B_{\frac{2}{\varepsilon}+||t||}]-t\subseteq T_{u'_n}T_tx=T_{u_n +t}x  ,$$ and  $$B_{\frac{1}{\epsilon}}\subseteq\operatorname{supp}(T_{u_n}x[B_{\frac{2}{\varepsilon}+||t||}]-t).$$
    Consequently, we have that
     \begin{equation}\label{igualdad_en_parche_grande}
         T_{v_n} x_n[B_{\frac{2}{\varepsilon}+||t||}]-t\in T_{v'_n+t_n}x_n[[B_\frac{1}{\epsilon}]],\; T_{u_n}x[B_{\frac{2}{\varepsilon}+||t||}]-t\in T_{u'_n+t}[[B_\frac{1}{\epsilon}]].
    \end{equation}

     By \eqref{Consecuencia_convergencia}, we know that
   $$T_{v_n} x_n[B_{\frac{2}{\varepsilon}+||t||}]-t=T_{u_n}x[B_{\frac{2}{\varepsilon}+||t||}]-t.$$Using \eqref{igualdad_en_parche_grande} we deduce that, in particular, $T_{v'_n+t_n}x_n[B_{\frac{1}{\epsilon}}]=T_{u'_n+t}x[B_{\frac{1}{\epsilon}}]$, so
   we conclude \eqref{convergencia_tilings_trasladados}.
\end{proof}

\begin{remark}
For a full tiling made up of congruence prototiles, We define a metric $d_2$. This metric makes two tilings close if, after an isometry, they agree on a large ball centered at the origin.

We first define a metric for the group of direct isometries of $\R^d$,

$$
 \mathcal{E}^d := \{ Tp = Ap + b : A \in SO(d), b \in \R^d \},
$$

given by

$$
d_{\mathcal{E}^d} (A_1 p + b_1 , A_2 p + b_2) := \max \{ ||A_1 - A_2|| , ||b_1 - b_2|| \}.
$$

 Then, denoting

$$
B_r (\mathcal{E}^d) = \{ T \in \mathcal{E}^d : d_{\mathcal{E}^d} (T,Id) < r \} = \{ Tp = Ap + b : ||A-Id|| < r, ||b|| < r \}.
$$

    We define a metric on the full tiling space $X_\mathcal{T}$ as

    $$
    d_2(x,y) = \min\left(\dfrac{1}{\sqrt{2}},\tilde{d}_2(x,y) \right),
    $$
    where $\tilde{d}_2(x,y)$ is defined as

    $$\inf  \left\{ 0<r<\dfrac{1}{\sqrt{2}} : \exists T_1,T_2 \in B_r(\mathcal{E}^d), \text{ such that }  T_1x[B_\frac{1}{r}] = T_2y[B_\frac{1}{r}] \right\}.$$

The proof that this function defines a metric is similar to the used for $d_1$.

Analogously, a full tiling space has finite local complexity under the Euclidean group if and only if is compact with $d_2$. Also, the action of $\mathcal{E}^d$ is continuous. The proofs follow the same scheme as the ones used for a full tiling made up of prototiles.
\end{remark}

\section{Tilings as topological dynamical systems}

\subsection{Abstract topological dynamical systems}
\begin{definition}
A \emph{topological dynamical system} is a triple $(X,G,T)$, where
$G$ is a topological group, $X$ is a topological space, and
\[
T \colon G \times X \longrightarrow X
\]
is a continuous action of $G$ on $X$. That is,
\[
T(e,x)=x
\qquad\text{and}\qquad
T(g,T(h,x))=T(gh,x)
\]
for every $g,h\in G$ and $x\in X$. If $G$ carries no topology, then $G$ will be provided by its discrete topology. For $g\in G$, $T_g$ denotes the map from $X$ to $X$ such that given $x\in X$ it returns $T(g,x)$, note that $T_g$ is a homeomorphism for any $g\in G$.

 The \emph{orbit} of a point $x\in X$ is the set
\begin{equation*}
Gx:=\{T(g,x): g\in G\},
\end{equation*}

and we may also denote $Gx$ by $\orbit{x}$. For $H\subseteq G$, we denote $Hx:=\{T(h,x):h\in H\}$. We may also denote $Gx$ by $T(x)$.

A set $Y\subseteq X$ is said to be $G$-invariant if $T_g(Y)=Y$ for every $g\in G$, or equivalently, if $T_g(Y)\subseteq Y$ for every $g\in G$.

 A \emph{subsystem} is a nonempty closed subset $Y\subseteq X$ that is $G$-invariant. In that case, the restriction
\begin{equation*}
T|_{G\times Y}G\times Y\longrightarrow Y
\end{equation*}
then defines a topological dynamical system $(G,Y,T|_{G\times Y})$. Observe that given $x \in X$, $\overline{Gx}$ is a subsystem.

A point $x\in X$ is called \emph{transitive} if its orbit is dense, i.e, $\overline{Gx}=X$. We say that $(G,X,T)$ is point transitive if there is a transitive point in $X$.

The topological dynamical system $(G,X,T)$ is said to be \emph{minimal} if every point in $X$ is transitive. Equivalently, $(G,X,T)$ is minimal if it has no proper subsystem, that is, every subsystem is either empty or the whole space $X$. On the other hand, $(G,X,T)$ is said to be transitive if for every nonempty open sets $U,V\subseteq X$, there is $g\in G$ such that $T_g(U)\cap V\neq\emptyset$.

\end{definition}

\begin{theorem}\label{theo:trans_sistemas}\cite[Theorem 9.20 and Remark 9.19]{gottschalk1955topological}
Let $(G,X,T)$ be a topological dynamical system. If there is  point in $X$ that is transitive then $X$ is transitive. If $X$ is a separable and complete metric space, then the existence of a transitive point is equivalent to $X$ being transitive.    
\end{theorem}

A well known application of Zorn's lemma is that any compact topological dynamical system possesses a minimal subsystem, we include a proof here for the sake of completeness.

\begin{proposition}\label{prop:existencia_minimal}
    Any compact topological dynamical system $(G,X,T)$ has a minimal subsystem.
\end{proposition}

\begin{proof}
    Consider the set $\mathcal{S}:=\{Y\subseteq X: Y \text{ is a subsystem}\}$. $\mathcal{S}$ is partially ordered by inclusion. Let $\mathcal{C}\subseteq \mathcal{S}$ be a nonempty chain. Then, by the compactness of $X$ and the finite intersection property of $\mathcal{C}$, we deduce that $$\bigcap_{K\in\mathcal{C}}K\neq\emptyset.$$

    Also, $\bigcap_{K\in\mathcal{C}}K$ is closed and invariant, hence $\bigcap_{K\in\mathcal{C}}K\in\mathcal{S}$ and is a lower bound of $\mathcal{C}$. By Zorn's lemma, there is a minimal element of $\mathcal{S}$, and it is clear that this is a minimal subsystem.
\end{proof}

\begin{definition}
    A set $A\subseteq G$ will be called left syndetic if there is a compact set $K\subseteq G$ such that $G=AK$. If $G$ is abelian, then we simply use the term syndetic instead of left syndetic. A point $x\in X$ will be called almost periodic provided that for every neighborhood $U$ of $x$ there is a left syndetic set $A$ such that $Ax\subseteq U$. Equivalently, a point $x$ is almost periodic if for every neighborhood $U$ of $x$ the set $R(x,U):=\{g\in G: T_g(x)\in U\}$ is left syndetic.
\end{definition}

When $X$ is a regular topological space, a necessary condition for a point $x\in X$ to be almost periodic is $\overline{Gx}$ to be minimal.
\begin{theorem}\label{theo:orbita_minimal}\cite{gottschalk1946almost}

Assume $X$ to be a regular topological space. For $x\in X$ to be almost periodic is necessary that $\overline{Gx}$ is minimal as a dynamical system. If $\overline{Gx}$ is compact, then this condition is also sufficient.

\end{theorem}

\subsection{Tiling spaces and dynamics}

Consider a full tiling space $X_{\mathcal{T}}$ with its topology under the metric $d_1$. Denote by $T$ the action by translation of $\R^d$ on $X_\mathcal{T}$, that is, $T\colon \R^d\times X_{\mathcal{T}}\to X_{\mathcal{T}}$ is given by $T(v,x):=T_vx$, then $(X_{\mathcal{T}},\R^d,T)$ is a topological dynamical system. From now on, finite local complexity will mean finite local complexity under translation, although the results of this section may be adapted to finite local complexity under the euclidean group just by reproducing the proofs.

For $Y\subseteq X_{\mathcal{T}}$ and $n\in\N$, we denote by $\mathcal{T}^{(n)}(Y)$ the set of protopatches consisting of $n$ tiles such that a translation of its representative appears in some tiling of $Y$, that is $$\mathcal{T}^{(n)}(Y):=\{[x']\in\mathcal{T}^{(n)}:\exists y'\in [x'],\exists y\in Y, \; y'\subseteq y\}.$$

For a singleton we will denote $\mathcal{T}^{(n)}(\{x\})$ just by $\mathcal{T}^{(n)}(x)$
\begin{definition}
    We will say that $Y\subseteq X_{\mathcal{T}}$ has finite local complexity if $\mathcal{T}_{Y}^{(2)}$ is finite. Equivalently, $Y$ has finite local complexity if $\mathcal{T}_{Y}^{(n)}$ is finite for every $n\in\N$, that these two statements are equivalent is analogous to Propositon \ref{prop: flc}. In particular, we will say that a tiling $x\in X_\mathcal{T}$ has finite local complexity if $\{x\}\subseteq X_\mathcal{T}$ has finite local complexity.
\end{definition}

Reproducing the same proof of Theorem \ref{theo: compacidad d1}, we deduce the following.

\begin{theorem}\label{theo:compacidad_flc_subespacios}
    Let $Y\subseteq X_\mathcal{T}$. If $Y$ has finite local complexity, then $Y$ is relatively compact. On the other hand, if $Y$ is invariant under $\R^d$ and is relatively compact, then it has finite local complexity.
\end{theorem}

\begin{definition}\label{def:tiling_space}
    A tiling space is a subset of $X_{\mathcal{T}}$ that is closed and invariant under translation ($\R^d$-invariant), i.e, a subsystem of $X_\mathcal{T}$. If $Y\subseteq X_{\mathcal{T}}$ is a tiling space, we call the pair $(Y,T)$ a tiling dynamical system.
\end{definition}

In fact, tiling spaces have a very special form, one way of constructing them is by a set of forbidden protopatches.

\begin{definition}
    Let $X_{\mathcal{T}}$ be a full tiling space and let $\mathcal{F}\subseteq\mathcal{T}^{*}$. We define $$X_{\setminus \mathcal{F}}:=\{x\in X_{\mathcal{T}}: \text{ no patch of } x \text{ belongs to a protopatch of } \mathcal{F}\}.$$

    $\mathcal{F}$ is the set of forbidden protopatches of $X_{\setminus \mathcal{F}}$, thus we call such a set $\mathcal{F}$ a set of forbidden protopatches.
\end{definition}

\begin{proposition}
    Let $\mathcal{F}\subseteq \mathcal{T}^*$, then $X_{\setminus \mathcal{F}}$ is a tiling space. Reciprocally, for any tiling space $Y\subseteq X_{\mathcal T}$ there is a set $\mathcal{F}\subseteq \mathcal{T}^*$ of forbidden protopatches such that $Y=X_{\setminus \mathcal{F}}$.
\end{proposition}

\begin{proof}
    Consider $X_{\setminus \mathcal{F}}$, let $v\in\R^d$ and $x\in X_{\setminus \mathcal{F}}$. It is clear that $T_vx \in X_{\setminus \mathcal{F}}$ since every patch of $T_vx$ is equivalent to a patch in $x$ and no patch of $x$ is in a protopatch of $\mathcal{F}$, hence $X_{\setminus \mathcal{F}}$ is $T$-invariant. On the other hand, consider a sequence $(x_n)_{n\in\N}$ in $X_{\setminus \mathcal{F}}$ convergent to $x\in X_{\mathcal{T}}$. Let $x'\subseteq x$ be a patch and  let $r>0$ be large enough so that $\operatorname{supp}(x')\subseteq B_r$. Consider $\varepsilon\in(0,\frac{\sqrt{2}}{2})$ small enough so that $r<\frac{1}{\epsilon}-\epsilon$, because of $x_n\to x$ we deduce that there is $m\in\N$ and $v,u\in B_\epsilon$ such that $T_{v}x_{m}[B_\frac{1}{\epsilon}]=T_ux[B_\frac{1}{\epsilon}]$. Observe that $x'-u\subseteq T_ux[B_\frac{1}{\epsilon}]\subseteq T_vx_{m}$ thus $x_{m}$ contains a translation of $x'$, and since $x_{m}\in X_{\setminus \mathcal{F}}$, we deduce that $x'$ is not in a protopatch of $\mathcal{F}$. Thence, by the arbitrariness of $x'$, we conclude that $x\in X_{\setminus \mathcal{F}}$ and $X_{\setminus \mathcal{F}}$ is a tiling space.

    For the converse, assume that $Y$ is a tiling space and consider its set of allowed protopatches $$\mathcal{A}:=\{x'\in\mathcal{T}^*: \exists x\in Y, \exists s\in\R^d, x'-s\subseteq x \}.$$

    We claim that $Y= X_{\setminus\mathcal{A}^c}$. It is clear that $Y\subseteq X_{\setminus\mathcal{A}^c}$. For the other inclusion, let $y\in X_{\setminus\mathcal{A}^c}$. By definition, for every $n\in\N$ we have that $y[B_n]$ is not in a protopatch of $\mathcal{A}^c$, thus there is $y_n\in Y$ and $s_n \in\R^d$ such that $y[B_n]-s_n \subseteq y_n$. Using that $Y$ is a tiling space we deduce that $T_{-s_n}y_n \in Y$. Also, we have that $y[B_n] \subseteq T_{-s_n}y_n$ for every $n\in \N$. Therefore $T_{-s_n}y_n\to y$. Consequently, $y\in Y$ because $Y$ is closed.
\end{proof}

\begin{corollary}
    A tiling space $Y\subseteq X_{\mathcal{T}}$ has finite local complexity if and only if is a compact.
\end{corollary}

\begin{proof}
    Theorem \ref{theo:compacidad_flc_subespacios}.
\end{proof}

It is clear that if the full tiling $X_\mathcal{T}$ has finite local complexity then every tiling space $Y\subseteq X_{\mathcal{T}}$ is compact.

An important example of tiling space is the closure of the orbit of a tiling, sometimes also called the hull of the corresponding tiling. Observe that the set of forbidden protopatches of the hull of a tiling is the set of all protopatches that are not the protopatch of any patch in $x$. 

When does the hull of a tiling have finite local complexity depends only on the respective tiling having finite local complexity.

\begin{proposition}\label{prop:flc_hull}
    Let $x\in X_\mathcal{T}$ be a tiling. $\cl{\orbit{x}}$ has finite local complexity if and only if $x$ has finite local complexity.
\end{proposition}

\begin{proof}
    We will see that $\mathcal{T}^{(2)}(\cl{\orbit{x}})=\mathcal{T}^{(2)}(x)$. It is clear that $\mathcal{T}^{(2)}(x)\subseteq \mathcal{T}^{(2)}(\cl{\orbit{x}})$ because $x\in \cl{\orbit{x}}$. Let $y'$ be a patch of two tiles such that $[y']\in \mathcal{T}^{(2)}(\cl{\orbit{x}})$, then there is a tiling $y\in\cl{\orbit{x}}$ such that $y$ contains a translation of $y'$, let $z'$ be such translation. Since $y\in \cl{\orbit{x}}$, there is a sequence $(y_n)_{n\in\N}$ in $\orbit{x}$ such that $y_n \to y$. Let $r>0$ such that $\operatorname{supp}(z')\subseteq B_r$. Let $\varepsilon\in(0,\frac{\sqrt{2}}{2})$ such that $r<\frac{1}{\epsilon}-\epsilon$. Using that $y_n \to y$, there is $n_0$ large enough and there is $u,v\in B_\epsilon$ such that $T_{v}y_{n_0}[\frac{1}{\epsilon}]=T_uy[b_\frac{1}{\epsilon}]$. Observe that $z'-u\subseteq T_uy[B_\frac{1}{\epsilon}]$, therefore $z'\subseteq T_{v-u}y_{n_0}$ That is, $y_{n_0}$ contains a translation of $y'$. Therefore, $x$ contains a translation of $y'$, because $y_{n_0}$ is a translation of $x$.
\end{proof}

In section \ref{sec:tiling_topology} we saw that under the tiling metric \eqref{def_metrica} compactness of a tiling space, which is a topological concept, is equivalent to a ``tiling property", which in this case is finite local complexity. Now, we will see under what ``tiling property" is the hull of a tiling minimal, which is a dynamical concept.

\begin{definition}
    A tiling $x\in X_\mathcal{T}$ is repetitive if for every patch $x'\subseteq x$ there is $r>0$ such that for every $t\in\R^d$ there is $s\in\R^d$ with $x'-s\subseteq x$ such that $\operatorname{supp}(x'-s) \subseteq B(t,r).$ That is, for every patch $x'\subseteq x$ there is a radius $r>0$ such that for every ball of radius $r$ there is a translation of $x'$ in $x$ whose support is contained  in such ball.
\end{definition}

A repetitive tiling is a tiling such that any patch of $x$ is being locally repeated to the ``horizon" with a bounded gap between occurrences.

Observe that in the context of $\R^d$, a set $A\subseteq\R^d$ is syndetic if and only if there is $r>0$ such that every ball of radius $r$ intersects $A$.

\begin{theorem}\label{theo:repetitivo=minimal}
    Let $x\in X_\mathcal{T}$ be a tiling. If $x$ is repetitive then $\cl{\orbit{x}}$ is minimal. If $x$ has finite local complexity, then $\cl{\orbit{x}}$ being minimal implies that $x$ is repetitive.
\end{theorem}

\begin{proof}
    Assume $x$ to be repetitive. To show that $\cl{\orbit{x}}$ is minimal we will  use Theorem \ref{theo:orbita_minimal}. Consider a basic neighborhood of $x$, that is, $B(x,\delta)$ with $\delta>0$. We shall see that $R(x,B(x,\delta))$ is syndetic. Since $x$ is repetitive, there is $r>0$ such that for any ball $B(t,r)$ there is a translation of $x[B_{\frac{1}{\delta}}]$ in $x$ with its support contained in $B(t,r)$. We will see that any ball of radius $r$ intersects $R(x,B(x,\delta))$. Consider $t\in\R^d$ and $B(t,r)$. We know that there is $s\in\R^d$ with $x[B_{\frac{1}{\delta}}]-s\subseteq x$ such that
    \begin{equation}\label{caso_trivial_minimalidad}
        \operatorname{supp}(x[B_{\frac{1}{\delta}}]-s)\subseteq B(t,r).
    \end{equation}
    Note that since $0\in\operatorname{supp}(x[B_{\frac{1}{\delta}}])$, then, using \eqref{caso_trivial_minimalidad}, we deduce that $-s \in \operatorname{supp}(x[B_{\frac{1}{\delta}}]-s)$, thus $$-s\in B(t,r).$$ According to $x[B_{\frac{1}{\delta}}]\subseteq T_{-s}x$, we have that $$T_{-s}x\in B(x,\delta).$$ Therefore $-s\in B(t,r)\cap R(x,B(x,\delta))$, consequently, $R(x,B(x,\delta))$ is syndetic and by Theorem \ref{theo:orbita_minimal}, $\cl{\orbit{x}}$ is minimal.
    
    For the reverse implication, we assume that $x$ has finite local complexity and we proceed by contradiction, so let us assume that $\cl{\orbit{x}}$ is minimal but $x$ is not repetitive. Since $x$ has finite local complexity, by Theorem \ref{theo:compacidad_flc_subespacios} and Proposition \ref{prop:flc_hull}, we deduce that $\cl{\orbit{x}}$ is compact. Since $x$ is not minimal, there is a patch $x'\subseteq x$ such that
    \begin{equation}\label{rumbo a contradiccion}
        \forall n>0,\exists t_n\in\R^d, \forall s\in\R^d \text{ with $x'-s \in x$}, \overline{\operatorname{supp}(x'-s)\subseteq B(t_n,n)}.
    \end{equation}
    In particular, $\operatorname{supp}(x')$ is not contained in $B(t_n ,n)$.

    Consider the squence $(T_{t_n}x)_{n\in\N}$ in $\cl{\orbit{x}}$. By compactness we know that there is a stricly increasing map $\varphi\colon \N\to\N$ such that $T_{t_{\varphi(n)}}x$ converges to $y\in \cl{\orbit{x}}$. We claim that $y$ has no translation of $x'$. Indeed, assume that there is a translation $y'$ of $x'$ in $y$. Consider,  $N\in\N$ large enough so that 
    \begin{equation*}
        \operatorname{supp}(y')\subseteq B_{N-\frac{1}{N}}. 
    \end{equation*}
    Since $T_{t_{\varphi(n)}}x\to y$, there is $n_0 \in\N$ such that for every $n\geq n_0$,
    \begin{equation*}
         \exists u_n,v_n\in B_\frac{1}{N},\; T_{v_n+t_{\phi(n)}}x[B_{N}]=T_{u_n}y[B_{N}].
    \end{equation*}
    In particular, since $y'-u_n\subseteq T_{u_n}y[B_{N}]$, it follows that 
    \begin{equation}\label{lol}
        y'+v_n -u_n \subseteq T_{t_{\phi(n)}}x.
    \end{equation}

    Since $\operatorname{supp}(y'-u_n)\subseteq B_N$ and $||v_n||\leq\frac{1}{N}$, we obtain that 
    \begin{equation}\label{patch_en_el_limite}
        \operatorname{supp}(y'-u_n+v_n)\subseteq B_{N+1}.
    \end{equation}
    Consider $n:= \max\{n_0 ,N+1\}$, then, by \eqref{patch_en_el_limite}, we derive that 
    \begin{equation}\label{hola}
        \operatorname{supp}(y'-v_n+u_n)\subseteq B_{N+1} \subseteq B_{n}\subseteq B_{\varphi(n)}.
    \end{equation}
     Hence, by \eqref{lol} and \eqref{hola} we deduce that $y'-v_n+u_n+t_{\phi(n)}\subseteq x$  and $$\operatorname{supp}(y'-v_n+u_n+t_{\phi(n)})\subseteq B(t_{\varphi(n)},\varphi(n)).$$ Thence, recalling that $y'$ is a translation of $x'$, we conclude that there is a translation of $x'$ in $x$ whose support is contained in $B(t_{\varphi(n)},\varphi(n))$, which contradicts \eqref{rumbo a contradiccion}. Consequently, $x$ must be repetitive.

\end{proof}

Note that when $x\in X_{\mathcal{T}}$ has finite local complexity and is repetitive, then every tiling of $\cl{\orbit{x}}$ has finite local complexity and is repetitive. A tiling space that is minimal is the hull of any of its tiling, hence, if a minimal tiling space has finite local complexity then any of its tilings is repetitive. Furthermore, they will be ``uniformly repetitive", and this will characterize the minimality of a tiling space.

\begin{definition}
    We will say that a tiling space $Y$ is uniformly repetitive if for every tiling $x\in Y$ and for every patch $x'\subseteq x$, there is $R>0$ such that for every tiling $y\in Y$ and for every $t\in\R^d$ there is a translation of $x'$ in $y$ whose support is contained in $B(t,R)$.
\end{definition}

\begin{proposition}\label{prop:tiling_space_minimal}
    Let $Y\subseteq X_{\mathcal{T}}$ be a tiling space. If $Y$ is uniformly repetitive then $Y$ is minimal. Conversely, if $Y$ is minimal and has finite local complexity, then it is uniformly repetitive.
\end{proposition}

\begin{proof}
    Assume $Y$ is minimal and has finite local complexity, consider $x\in Y$ and $x'\subseteq x$. By minimality, we know that $Y= \cl{\orbit{x}}$, thence, by Theorem \ref{theo:repetitivo=minimal}, $x$ is repetitive. It follows that there is $M>0$ such that any ball of radius $M$ contains the support of a copy of $x'$ in $x$. We claim that is enough to consider $R:= M+1$. Indeed, Let $y\in Y$, and consider $t\in\R^d$. Consider $r<\frac{\sqrt{2}}{2}$ small enough such that $B(t,M+1) \subseteq B_{\frac{1}{r}}$. Using that $y\in Y= \cl{\orbit{x}}$, we obtain that there is $w\in\R^d$ such that $d_1(T_wx, y)<r$, thence $$\exists u,v \in B_r,\; T_{v+w}x[B_{\frac{1}{r}}]= T_uy[B_\frac{1}{r}]. $$

    On the other hand, there is a translation of $x'$ in $x$, say $z'$, such that $\operatorname{supp}(z')\subseteq B(t+v+w,M)$, thus $z'-v-w \in T_{v+w}x$ and $\operatorname{supp}(z'-v-w)\subseteq B(t,M)\subseteq B_{\frac{1}{r}}$. Consequently, $z'-v-w\subseteq T_{v+w}x[B_\frac{1}{r}]$, hence $z'-v-w \subseteq T_uy[B_{\frac{1}{r}}]\subseteq T_{u}y$, therefore $z'+u-v-w\subseteq y$ and since $||u||< 1$, we deduce that $\operatorname{supp}(z'+u-v-w)\subseteq B(t,M+1)$, concluding that $Y$ is uniformly repetitive.

    For the converse, assume that $Y$ is uniformly repetitive. Consider $x\in Y$, we shall see that $\cl{\orbit{x}}= Y$. let $y\in Y$ and let $\varepsilon \in (0,\frac{\sqrt{2}}{2})$. Since $Y$ is uniformly repetitive, in particular, there is $s\in\R^d$ such that $y[B_{\frac{1}{\varepsilon}}]-s\subseteq x$. Therefore $y[B_{\frac{1}{\varepsilon}}]\subseteq T_{-s}x$, hence $d_1(T_{-s}x,y)\leq \varepsilon$. Since $\varepsilon \in (0,\frac{\sqrt{2}}{2})$ was arbitrary, we conclude that $y\in \cl{\orbit{x}}$.

\end{proof}

\begin{corollary}
    If $x\in X_{\mathcal{T}}$ has finite local complexity, then
    $$\cl{\orbit{x}} \text{ is minimal}\Longleftrightarrow x \text{ is repetitive}\Longleftrightarrow x \text{ is almost periodic.}$$
\end{corollary}

\begin{proof}
    Consequence of Theorem \ref{theo:compacidad_flc_subespacios} together with Theorem \ref{theo:orbita_minimal} and Theorem \ref{theo:repetitivo=minimal}
\end{proof}

\begin{corollary}
    If $\mathcal{T}$ admits a tiling $x\in X_{\mathcal{T}}$ with finite local complexity, then it admits a repetitive tiling.
\end{corollary}

\begin{proof}
    By Proposition \ref{theo:compacidad_flc_subespacios}, the hull of $x$ is a compact dynamical system, thence it admits a minimal subsystem $Y$. Let $y\in Y$, by minimality $\cl{\orbit{y}}=Y$, thus $\cl{\orbit{y}}$ is minimal. Thus, by Theorem \ref{theo:repetitivo=minimal}, $y$ is repetitive. 
\end{proof}

Regarding transitivity, the following is a direct characterization.
\begin{proposition}\label{prop:trans}
    A tiling space $Y\subseteq X_{\mathcal{T}}$ is transitive if and only if for every $y_1,y_2\in Y$ and for each pair of patches $y'_{1}\subseteq y_1, y'_2\subseteq y_2$ there is $x\in Y$ such that $x$ contains a translation of $y'_1$ and a translation of $y'_2$.
\end{proposition}

\begin{proof}
    Assume that $Y$ is transitive. Let $y_1, y_2\in Y$, let $y'_1$ be a patch of $y_1$ and let $y'_2$ be a patch of $y_2$. Let $R_1,R_2>0$ be such that $\operatorname{supp}(y'_1)\subseteq B_{\frac{1}{R_1}-R_1}$ and $\operatorname{supp}(y'_2)\subseteq B_{\frac{1}{R_2}-R_2}$. Consider the open sets $U:=B(y_1,R_1)$ and $V:=B(y_2,R_2)$. By transitivity, there is $s\in \R^d$ such that $T_s(U)\cap V\neq\emptyset$. Let $x\in T_s(U)\cap V$. Then there are $t_1,u_1\in B_{R_1}$ such that $T_{t_1-s}x[B_\frac{1}{R_1}]=T_{u_1}y[B_{\frac{1}{R_1}}]$ and there are $t_2 ,u_2\in B_{R_2}$ such that $T_{t_2}x[B_\frac{1}{R_2}]=T_{u_2}y[B_{\frac{1}{R_2}}]$. Therefore, since $y'_1 -u_1\subseteq T_{u_1}y[B_{\frac{1}{R_1}}]$ and $y'_2 -u_2 \subseteq T_{u_2}y[B_{\frac{1}{R_2}}]$ it follows that $y'_1 -u_1 +t_1-s \subseteq x$ and $y'_2 -u_2 +t_2 \subseteq x$.
    
    For the reverse implication, 
    let $U,V\subseteq X_{\mathcal{T}}$ be nonempty open sets. Without loss of generality, we assume that $U=B(z_1,r_1)$ and $V=B(z_2,r_2)$ with $z_1,z_2\in Y$ and $r_1,r_2>0$ (each ball is restricted to $Y$). By hypotheses, there is $z\in Y$ and $v_1,v_2\in \R^d$ such that $z_1[B_{\frac{1}{r_1}}]-v_1\subseteq z$ and $z_2[B_{\frac{1}{r_2}}]-v_2\subseteq z$. Thus, we deduce that  $T_{-v_2}z\in T_{v_1-v_2}(U)\cap V\neq\emptyset$, and $T_{-v_2}z \in Y$ because $Y$  is a tiling space.
\end{proof}

\begin{corollary}
    If $Y\subseteq X_{\mathcal{T}}$ is a transitive tiling space  with finite local complexity, then there is $x\in Y$ such that $Y=\cl{\orbit{x}}$.
\end{corollary}

\begin{proof}
    Theorem \ref{theo:compacidad_flc_subespacios} and Theorem \ref{theo:trans_sistemas}.
\end{proof}

We end with a simple example of a transitive  tiling space that is not minimal.
\begin{example}
    
Consider $\mathcal{S}_n$, $n>1$, to be the set of $1\times 1$ square prototiles
marked with colors $1,2,\ldots,n$. To do this we also need to
modify our notion of equivalence so that differently colored squares are not
considered to be equivalent. Consider the tiling space $C_n\subseteq X_{\mathcal{S}_n}$ of all tilings made up from the prototiles $\mathcal{S}_n$ where tiles are allowed only to meet only full edge to edge, i.e, we are forbidding patches where squares do not meet full edge to edge. 

If $c_1,c_2\in C_n$ are two tilings, and $c'_1,c'_2$ are two patches of $c_1$ and $c_2$, respectively, there is a tiling in $C_n$ that contains translations of both patches; is enough to modify $c_1$ away from $c'_1$ and to a pattern in $c_1$ equal to the patch $c'_2$ without considering color, change its coloring so that it coincides with $c'_2$ up to translation. Hence, by Proposition \ref{prop:trans}, $C_n$ is transitive. On the other hand, by Proposition \ref{prop:tiling_space_minimal}, $C_n$ is not minimal, because a tiling of $C_n$ made up only of squares of color 1 does not contain any patch that contains a square colored with 2, thence $C_n$ is not uniformly repetitive.
\end{example}

\bibliographystyle{acm}
\bibliography{references}

\end{document}